\pdfoutput=1
\documentclass{jgcc}

\usepackage{lastpage}
\jgccdoi{16}{1}{6}{13875}
\jgccheading{}{\pageref{LastPage}}{}{}{May~11,~2024}{Jul.~29,~2024}{}  
\usepackage{hyperref} 
\usepackage{amsmath,amssymb}
\usepackage{algorithmic}

\DeclareMathOperator{\LC}{LC}

\DeclareMathOperator{\LM}{LM}
\DeclareMathOperator{\NF}{NF}
\DeclareMathOperator{\Mat}{Mat}
\DeclareMathOperator{\lcm}{lcm}
\DeclareMathOperator{\Syz}{Syz}
\DeclareMathOperator{\Rad}{Rad}

\newcommand\ssum{\textstyle\sum\limits}

\newcommand\Z{\mathbb{Z}}

\let\rho=\varrho
\let\epsilon=\varepsilon

\title{Computing the Unit Group of a Commutative Finite \texorpdfstring{$\mathbb{Z}$}{Z}-Algebra}

\author{Martin Kreuzer}
\address[Martin Kreuzer]{Fakult\"{a}t f\"{u}r Informatik und Mathematik \\
Universit\"{a}t Passau, D-94032 Passau, Germany}
\email{martin.kreuzer@uni-passau.de}

\author{Florian Walsh}
\address[Florian Walsh]{Fakult\"{a}t f\"{u}r Informatik und Mathematik \\
Universit\"{a}t Passau, D-94032 Passau, Germany}
\email{florian.walsh@uni-passau.de}

\thanks{2020 \textit{Mathematics Subject Classification.} 68W30, 20F05, 13P99, 16Z05.}

\begin{document}

\begin{abstract}
For a commutative finite $\mathbb{Z}$-algebra, i.e., for a commutative ring $R$ whose additive group
is finitely generated, it is known that the group of units of $R$ is finitely generated, as well.
Our main results are algorithms to compute generators and the structure of this group.
This is achieved by reducing the task first to the case of reduced rings, then to torsion-free reduced rings,
and finally to an order in a reduced ring. The simplified cases are treated via a calculation of 
exponent lattices and various algorithms to compute the minimal primes, primitive idempotents, and other basic
objects. All algorithms have been implemented and are available as a {\tt SageMath} package. Whenever possible,
the time complexity of the described methods is tracked carefully.  
\end{abstract}

\keywords{finite Z-algebra, unit group, exponent lattice, explicitly given algebra}

\maketitle

%%%%%%%%%%%%%%%%%%%%%%%%%%%%%%%%%%%%%
%
%  Section 1: Introduction
%
%%%%%%%%%%%%%%%%%%%%%%%%%%%%%%%%%%%%%

\section{Introduction}

In the study of the structure of a commutative ring, one important aspect is its group of units.
A famous result in this direction dating back to 1846 is L.G.\ Dirichlet's unit theorem (see~\cite{dir1846})
which says that the group of units of the ring of integers of a number field is a finitely generated
abelian group. Much later, in 1972, this was generalized to orders in such rings by H.\ Zassenhaus (see~\cite{zas1972}).
With the advance of computer algebra, the computation of an actual system of generators of such a
unit group and its set of relations have become feasible, and algorithms achieving these tasks
have been developed (see, for example, \cite{cohen1997subexponential, biasse2014subexponential}).
Also for other types of rings, for which the group of units is known to be finitely generated,
explicit algorithms for computing their generators or their presentations have been described,
including for orders in (not necessarily commutative) finite dimensional $\mathbb{Q}$-algebras
(see~\cite{braun2015computing}), for integral group rings over finite abelian groups (see~\cite{faccin2013computing}), 
and for the affine coordinate rings of rational normal curves and elliptic curves (see \cite{chen2021computing}).

In this paper we improve on many of these results and consider the general case of a commutative finite $\Z$-algebra,
i.e., a commutative ring which is a finitely generated $\Z$-module. In~\cite{samuel1970algebraic}, P.\ Samuel 
proved that the unit groups of such rings are finitely generated. The main results of this paper are algorithms
for computing a system of generators of these unit groups, as well as for calculating their structure.

Let us describe the path we follow to reach these goals. After recalling some basic results about finite $\Z$-algebras
in Section~2, we devise algorithms for computing {\it exponent lattices} in such a ring~$R$, i.e.,
lattices of the type $\Lambda = \{ (a_1,\dots,a_k) \in \Z^k \mid f_1^{a_1} \cdots f_k^{a_k} = 1\}$
where $f_1,\dots,f_k\in R$. Using several techniques from our previous paper~\cite{kreuzer2024binomial},
we reduce the calculation of exponent lattices to the cases of 0-dimensional algebras over the fields~$\mathbb{Q}$
and~$\mathbb{F}_p$ with a prime~$p$. For this characteristic~$p$ part, we solve the problem first modulo~$p$
and then refine the answer modulo higher powers of~$p$ with a method resembling the well-known technique 
of Hensel lifting (see Proposition~\ref{prop:hensel_lifting}). 
Altogether, we obtain Algorithm~\ref{alg:exp_lattice_z-algebra}
and discuss several methods to improve its implementation (see Remark~\ref{rem:impl_exp_lattice}).

The next step is taken in Section~4 where we consider the case of a reduced finite $\Z$-algebra.
If the algebra is even integral, algorithms for computing its unit group are known (see Remark~\ref{remark:integral}).
The general case is treated by calculating the primitive idempotents and computing the unit
group of an order via reduction to the case of orders in number fields (see Lemma~\ref{lemma:order} 
and Algorithm~\ref{alg:unit_group_order}). This solves the torsion-free reduced case 
(see Corollary~\ref{alg:unit_group_torsion_free})
and allows us to deal with the general reduced case using a version of the Chinese Remainder Theorem
(see Lemma~\ref{lemma:reduction_to_torsion_free} and Algorithm~\ref{alg:units_reduced}).

Finally, in Section~5, we attack the general case of a finite $\Z$-algebra. The main additional
task is to find generators of $1+\Rad(0)$ (see Lemma~\ref{lemma:canonical_epi}). We provide two different solutions
(Algorithm~\ref{alg:unit_group_gens} and Lemma~\ref{lemma:rad_Macaulay}). All in all, 
we are able to compute a system of generators of the unit group of a finite $\Z$-algebra
(see Algorithm~\ref{alg:unit_group_gens}) and also its structure 
in terms of its rank and invariant factors (see Corollary~\ref{corollary:type}).

Throughout the paper we tried to keep track of the complexity of the presented algorithms.
First of all, this depends on the way the algebra~$R$ is given: either explicitly (via generators
and relations of $R^+$ plus the structure constants) or through a presentation $R=\Z[x_1,\dots,x_n]/I$
with an ideal~$I$ given by explicit generators. In the first case, many steps of the algorithms
can be performed in probabilistic polynomial time plus (possibly) one integer factorization.
In the second case, we may have to first compute a strong Gr\"obner basis to get going.
After we reduce everything to the case of an order in a number field, we have to rely on previous
work whose precise complexity estimates are apparently not known.

All algorithms in this paper are illustrated by explicit examples.
They were computed using an implementation by the second author in the software system
\texttt{SageMath}~\cite{sagemath}. The complete package is available freely from his GitHub
page~\cite{walsh2024binomial}. As for the basic definitions and notation, we adhere to
the terminology given in the books~\cite{kreuzer2000computational} and~\cite{kreuzer2005computational}.

\bigskip\bigbreak
%%%%%%%%%%%%%%%%%%%%%%%%%%%%%%%%%%%%%%%%%%%%%
%
%  Section 2: Preliminaries
%
%%%%%%%%%%%%%%%%%%%%%%%%%%%%%%%%%%%%%%%%%%%%%

\section{Preliminaries on Finite \texorpdfstring{$\Z$}{Z}-Algebras}

In this section we collect basic properties of finite $\mathbb{Z}$-algebras, i.e., $\mathbb{Z}$-algebras which
are finitely generated as a $\mathbb{Z}$-module. Given such an algebra~$R$, we denote its underlying
$\mathbb{Z}$-module by $R^+$. Subsequently, we assume that a $\Z$-algebra~$R$ is either given by an ideal~$I$ in
$P = \mathbb{Z}[x_1, \dots, x_n]$ such that $R=P/I$ or that it is given as follows.

\begin{rem}\label{remark:input}
A $\Z$-algebra $R$ is said to be \textbf{explicitly given} if it is given by the following information.
\begin{enumerate}
\item[(a)] A set of generators $\mathcal{G} = \{g_0, \dots, g_n\}$ of the $\Z$-module $R^+$, together with
a matrix $A = (a_{\ell k}) \in \Mat_{m,n+1}(\Z)$ whose rows generate the syzygy module
$\Syz_{\Z}(\mathcal{G})$ of $\mathcal{G}$.

\item[(b)] Structure constants $c_{ijk}\in\Z$ such that $g_i g_j = \sum_{k=0}^n c_{ijk} g_k$ for $i,j = 0, \dots, n$.
\end{enumerate}
Notice that we may assume $g_0 = 1$ and encode this information as an ideal
$$
I = \left\langle x_i x_j - \ssum_{k=0}^n c_{ijk} x_k, \; \ssum_{k=0}^n a_{\ell k} g_k
\mid i,j=1, \dots, n, \; \ell=1, \dots, m \right\rangle
$$
in $P = \Z[x_1, \dots, x_n]$ such that $R \cong P/I$.
\end{rem}

If $R=P/I$ is not explicitly given, then we can obtain an explicit representation from a strong Gr\"obner
basis of~$I$.

\begin{defi}
Given an ideal $I \subseteq P$ and a term ordering $\sigma$, a set of polynomials $G=\{g_1,\dots,g_r\}$ in~$I$
is called a \textbf{strong $\sigma$-Gr\"obner basis} of~$I$ if, for every non-zero polynomial $f\in I$,
there exists an index $i\in\{1,\dots,r\}$ such that $\LM_\sigma(f)$ is a multiple of~$\LM_\sigma(g_i)$.
\end{defi}

Strong Gr\"obner bases can be computed using a generalization of Buchberger's algorithm
(see for example \cite[Ch.~4]{adams1994introduction} or~\cite{eder2021standard}).
For some ideal-theoretic operations which can be performed effectively using
strong Gr\"obner bases, we refer to~\cite[Ch.~4]{adams1994introduction} and~\cite[Ch.~3]{kreuzer2000computational}.
Generators of the $\mathbb{Z}$-module~$R^+$ can be deduced from a strong Gr\"obner basis as follows.

\begin{prop}{\bf (Macaulay's Basis Theorem for Finite $\mathbb{Z}$-Algebras)}\label{lemma:macaulay}\ \\
Let $I \subseteq P$ be an ideal such that $P/I$ is a finite $\mathbb{Z}$-algebra, let~$\sigma$
be a term ordering on~$\mathbb{T}^n$, and let $L = \{m \in \LM_\sigma(I) \mid \LC_\sigma(m) = 1\}$ be the set of all monic
leading monomials of~$I$. Then the residue classes of the terms in
$\mathcal{O}_\sigma = \mathbb{T}^n \setminus L$ form a generating set of the $\mathbb{Z}$-module $P/I$.
\end{prop}

\begin{proof}
See Proposition~6.6 in \cite{kreuzer2024efficient}.
\end{proof}

Given generators of $R^+$ as in the preceding lemma, it is also possible to determine an explicit presentation
of~$R$ (see Algorithm 6.7 and Corollary 6.8 in \cite{kreuzer2024efficient}). From such a presentation
we can then determine the structure of $R^+$.

\begin{rem}
By the structure theorem for finitely generated modules over a principal ideal domain there exist
$r$ and $k_1, \dots, k_u$ in $\mathbb{N}$ such that $k_i$ divides $k_j$ for $i<j$ and such that
$$
R^+ \cong \mathbb{Z}^r \oplus \mathbb{Z}/k_1 \mathbb{Z} \oplus \cdots \oplus \mathbb{Z}/k_u \mathbb{Z}.
$$
The numbers $r$ and $k_1, \dots, k_u$ are uniquely determined by $R^+$. We call $r$ the \textbf{rank} and
$k_1, \dots, k_u$ the \textbf{invariant factors} of~$R^+$. The largest invariant factor $k_u$ is the exponent
of the torsion subgroup of $R^+$. We call it the \textbf{torsion exponent}~$\tau$ of~$R^+$.
\end{rem}

The rank, the invariant factors, and the torsion exponent can be determined using a Smith normal form computation
(for details we refer to Section~2 in~\cite{kreuzer2024efficient}). Algorithms which compute the Smith normal form
of an integer matrix can for example be found in~\cite{kannan1979polynomial} or~\cite{storjohann1996near}.
If an explicit presentation is given or has been determined from a strong Gr\"obner basis, many computations that
we need in the following sections can be performed efficiently, i.e., in (probabilistic) polynomial time in the bit
complexity of the input. More precisely, we have the following complexity results.

\begin{rem}\label{remark:complexity}
Assume that $R$ is an explicitly given finite $\mathbb{Z}$-algebra.
\begin{enumerate}
\item[(a)] The minimal prime ideals of~$R$ can be computed in zero-error probabilistic polynomial time
except for the factorization of one integer (see Algorithm~4.2 in~\cite{kreuzer2024efficient}).

\item[(b)] The primitive idempotents of $R$ can be obtained from its minimal prime ideals in polynomial time
using Algorithm~5.8 in~\cite{kreuzer2024efficient}.

\item[(c)] The intersection of ideals in $R$ can be determined in polynomial time using Proposition~2.9
in~\cite{kreuzer2024efficient}.

\end{enumerate}
\end{rem}

\bigskip\bigbreak
%%%%%%%%%%%%%%%%%%%%%%%%%%%%%%%%%%%%%%%%%%%%%
%
%  Section 3: Exponent Lattices
%
%%%%%%%%%%%%%%%%%%%%%%%%%%%%%%%%%%%%%%%%%%%%%

\section{Exponent Lattices in Finitely Generated \texorpdfstring{$\Z$}{Z}-Algebras}

Let $R=P/I$ be a finitely generated $\mathbb{Z}$-algebra. In the following we present an algorithm which computes
the multiplicative relations between units in~$R$. We emphasize that in this section we do not require 
that~$R$ is a finite $\mathbb{Z}$-algebra.

\begin{defi}
Let $R$ be a ring and let $f_1,\dots,f_k \in R^\times$. Then the lattice
$$
\Lambda = \{(a_1,\dots,a_k) \in \mathbb{Z}^k \mid f_1^{a_1} \cdots f_k^{a_k} = 1 \}
$$
is called the \textbf{exponent lattice} of $(f_1,\dots,f_k)$ in $R$.
\end{defi}

The goal of this section is to provide an algorithm which computes a basis of the exponent lattice of the tuple
$(f_1, \dots, f_k)$ in a finitely generated $\mathbb{Z}$-algebra. In the following we refer to this task simply as computing
an exponent lattice. Let us recall how this task is solved in affine $K$-algebras, i.e., in finitely generated
algebras over a field~$K$.

\begin{rem}{\bf (Computing Exponent Lattices in Affine $K$-Algebras)}\label{remark:exp_lattice_fields}\\
The problem of computing the exponent lattices has been considered by many authors.
For units in a number field algorithms can be found in~\cite{ge1993algorithms}, in
Section~7.3 of~\cite{kauers2005algorithms}, in Section~3 of~\cite{kauers2023order}, or
in~\cite{zheng2019effective}. Based on these algorithms, a method for computing the exponent lattice in
zero-dimensional $\mathbb{Q}$-algebras is presented in~\cite{lenstra2018algorithms}. Recently,
we generalized these results and presented a method (see Algorithm 5.3 in~\cite{kreuzer2024binomial}) for
computing exponent lattices in arbitrary affine $K$-algebras where $K$ is a field such that exponent lattices
in finite extensions of~$K$ can be effectively computed. Note, that this includes the cases $K = \mathbb{Q}$ and
$K = \mathbb{F}_p$.
\end{rem}

Now the main idea is to reduce the problem of computing an exponent lattice in~$R$ to computing exponent lattices
in affine $\mathbb{Q}$- and $\mathbb{F}_p$-algebras.

\begin{lem}\label{lemma:splitting}
Let $R$ be a ring, $I$ an ideal in $R$, and $f \in R$. If $m \in \mathbb{N}$ is such that $I : f^\infty = I : f^m$, then
\[
I = (I : f^m) \cap \langle I, f^m \rangle.
\]
\end{lem}
    
\begin{proof}
See \cite{greuel2008singular}, Lemma~3.3.6.
\end{proof}

Together with the following proposition this lemma is the main tool for reducing the exponent lattice computation
to $\mathbb{Q}$- and $\mathbb{F}_p$-algebras.

\begin{prop}\label{prop:gb_lcm}
Let $I$ be an ideal in $P$, let $G=\{g_1,\dots,g_s\}$ be a minimal strong Gr\"obner basis of~$I$, and
let $N\in \mathbb{Z}$ be the least common multiple of the leading coefficients of the elements of~$G$.
Then the following holds.
\begin{enumerate}
\item[(a)] $I = (I : \langle N\rangle) \cap (I+ \langle N\rangle)$

\item[(b)] If $I \cap \mathbb{Z} = \langle 0 \rangle$, then $I \mathbb{Q}[x_1, \dots, x_n] \cap P = I: \langle N\rangle$.

\end{enumerate}
\end{prop}

\begin{proof}
See \cite{kreuzer2023decomposing}, Proposition~4.3.
\end{proof}

The following example given in Section~4 of \cite{kreuzer2023decomposing} illustrates the fact that the least
common multiple $N$ of the leading coefficients of a strong Gr\"obner basis as in the proposition is in general not
the smallest number satisfying $I:\langle N\rangle = I: \langle N\rangle^\infty$.

\begin{exa}\label{example:gb_lcm}
Consider the ideal $I = \langle x^2,y^2,z^2,xz+yz,xy,2x-y, 3z \rangle \subseteq \mathbb{Z}[x,y,z]$.
The generators form a strong Gr\"obner basis of $I$ and the least common multiple of the leading coefficients
is~6. But we have $I : \langle 3 \rangle = I : \langle 6 \rangle = I : \langle 6 \rangle^\infty$.
\end{exa}

Let $\{g_1, \dots, g_k\}$ be a strong Gr\"obner basis of $I$ and let
$N$ be the least common multiple of the leading coefficients of the $g_i$. Then we have
$I = I : \langle N \rangle \cap \langle I, N \rangle$. The property
$I : \langle N \rangle = I \mathbb{Q}[x_1, \dots, x_n] \cap P$ then allows us to compute the exponent lattice
modulo $I : \langle N \rangle$. This can be done using Remark~\ref{remark:exp_lattice_fields}. The
ideal $\langle I, N \rangle$ can be further split into
$\langle I, N \rangle = \bigcap_{i=1}^r \langle I, p_i^{e_i} \rangle$ where $N = p_1^{e_1} \cdots p_r^{e_r}$ is
the prime factorization of $N$. Let $p \in \mathbb{N}$ be a prime number. The exponent lattice modulo an ideal of
the form $\langle I, p \rangle$ can be computed using the fact that the polynomial $f_1^{c_1}\cdots f_k^{c_k}-1$ is
in $I$ if and only if its canonical residue class is in $I \mathbb{F}_p[x_1, \dots, x_n]$. We can therefore again
apply Remark~\ref{remark:exp_lattice_fields}. It remains to handle ideals of the form $\langle I, p^e \rangle$
with $e>1$.

\begin{prop}\label{prop:hensel_lifting}
Let $I$ be an ideal such that $I \cap \mathbb{Z} = \{0\}$. Consider the finitely generated 
$\mathbb{Z}$-algebra $R = P/I$,
and let $f_1, \dots, f_k \in R^\times$. Let~$p$ be a prime number,
let $e$ be a positive integer, and let $b_1, \dots, b_m \in \mathbb{Z}^k$ be a basis of the exponent lattice
$\Lambda$ of $\left(\bar{f}_1, \dots, \bar{f}_k\right)$ in $P/\langle I, p^e \rangle$. 
Then the following conditions are equivalent.
\begin{enumerate}
\item[(a)] The $\mathbb{Z}$-linear combination $c = a_1 b_1 + \cdots + a_m b_m \in \Lambda$ with
                $a_1, \dots, a_m \in \mathbb{Z}$ is in the exponent lattice of
                $\left(\bar{f}_1, \dots, \bar{f}_k\right)$ in $P/\langle I, p^{e+1} \rangle$.

\item[(b)] The tuple $(a_1, \dots, a_m)$ is a solution of the linear equation over~$\mathbb{Z}$ in the
indeterminates $y_1, \dots, y_m$ given by
\begin{equation*}
                    \tag{i}\bar{h}_1 y_1 + \cdots + \bar{h}_m y_m = 0 \quad \text{ in } P/\langle I, p \rangle,
\end{equation*}
where $h_i = (f_1^{b_{i1}} \cdots f_k^{b_{ik}}-1)/p^e \in R$ and $\bar{h}_i$ is its residue class modulo
$\langle I, p \rangle$.

\end{enumerate}
\end{prop}

\begin{proof}
Since all tuples $b_i$ are in $\Lambda$ we have $f_1^{b_{i1}} \cdots f_k^{b_{ik}} = 1$ modulo
$\langle I, p^e \rangle$. Therefore there exists $g_i \in P$ such that
$f_1^{b_{i1}} \cdots f_k^{b_{ik}} = 1 + p^e g_i$ in $R$ for $i=1, \dots, m$. This shows $h_i = g_i$.
Now the tuple~$c$ is in the exponent lattice of $\left(\bar{f}_1, \dots, \bar{f}_k\right)$ in
$P/\langle I, p^{e+1} \rangle$ if and only if
$$
f^c = f^{a_1b_1} \cdots f^{a_mb_m} = (1+p^e g_1)^{a_1} \cdots (1+p^e g_m)^{a_m} = 1+a_1p^e g_1 + \cdots + a_m p^e g_m = 1
$$
in $P/\langle I, p^{e+1} \rangle$. This is equivalent to $a_1\bar{h}_1 + \cdots + a_m\bar{h}_m = 0$ in
$P/\langle I, p \rangle$, which is satisfied if and only if $(a_1, \dots, a_m)$ is a solution of
the linear equation~(i).
\end{proof}

At this point we are ready to compute the exponent lattice of $(f_1, \dots, f_k)$ modulo an ideal of the form
$\langle I, p^e \rangle$. This is achieved by first computing the exponent lattice of $(f_1, \dots, f_k)$ modulo
$I \mathbb{F}_p[x_1, \dots, x_n]$ and then iteratively applying Proposition~\ref{prop:hensel_lifting} to obtain
the exponent lattice modulo~$\langle I, p^i \rangle$ for $i=2, \dots, e$.

\begin{exa}
Let $I = \left\langle x^2+x+1, y^2+y+1, 8 \right\rangle$, and consider the finite $\mathbb{Z}$-algebra
$R=\mathbb{Z}[x,y]/I$. For $f_1 = 2x+1$, $f_2 = 4y+1$ and $f_3 = -2y-1$ let us compute the exponent lattice
of $\left(\bar{f}_1, \bar{f}_2, \bar{f_3}\right)$ in~$R$. To compute the exponent lattice modulo
$\langle I, 2 \rangle$ we form the zero-dimensional $\mathbb{F}_2$-algebra $\mathbb{F}_2[x,y]/I\mathbb{F}_2[x,y]$,
and obtain the exponent lattice $\Lambda_1 = \mathbb{Z}^3$. We then solve the linear equation in the
indeterminates $z_1, z_2, z_3$ given by
$x z_1 + 2y z_2 -(y+1)z_3 = 0$ modulo $\langle I, 2 \rangle$, and obtain the solution space
$M_1 = \langle (0,1,0), (2,0,0), (0,0,2) \rangle$. Since $\Lambda_1 = \mathbb{Z}^3$, this yields
$\Lambda_2 = M_1$. Then we solve the linear equation in the indeterminates $z_1, z_2, z_3$ given by
$yz_1 -z_2 -z_3 = 0$ modulo $\langle I, 2 \rangle$ and obtain the solution space
$M_2 = \langle b_1, b_2, b_3 \rangle$ with $b_1 = (0,1,1)$, $b_2 = (0,-1,1)$ and $b_3 = (2,0,0)$.
Finally, we compute the exponent lattice
$$
\Lambda_3 = \{c_1b_1 + c_2 b_2 + c_3 b_3 \mid c \in M_2\} = \langle (0,2,0), (2,0,2), (-2,0,2) \rangle.
$$
\end{exa}

Combining the previous results we now obtain the following algorithm.

\begin{algo}{\bf (Computing Exponent Lattices in Finitely Generated $\mathbb{Z}$-Algebras)}\label{alg:exp_lattice_z-algebra}\ \\
Let $R=P/I$ be a $\mathbb{Z}$-algebra, and let $f_1, \dots, f_k \in R^\times$. Consider the following sequence of
instructions.

\begin{algorithmic}[1]
   \STATE Compute $I \cap \mathbb{Z} = \langle q \rangle$.
   \IF {$q = 0$}
       \STATE Using Remark~\ref{remark:exp_lattice_fields}, compute the exponent lattice
            $\Lambda \subseteq \mathbb{Z}^k$ of $(f_1, \dots, f_k)$ in the $\mathbb{Q}$-algebra
            $\mathbb{Q} \otimes_{\mathbb{Z}} R$.
       \STATE Compute a strong Gr\"obner basis $\{g_1, \dots, g_\ell\}$ of $I$.
       \STATE Let $N = \lcm(\LC(g_1), \dots, \LC(g_\ell))$.
       \IF {$N = 1$}
          \RETURN $\Lambda$
       \ELSE
          \STATE Recursively apply the algorithm to compute the exponent lattice $M \subseteq \mathbb{Z}^k$
                 of $(f_1, \dots, f_k)$ in $P/(I + \langle N \rangle)$.
          \RETURN $\Lambda \cap M$
       \ENDIF
   \ELSE
      \STATE Compute the prime factorization $q = p_1^{e_1} \cdots p_r^{e_r}$.
      \FOR {$i=1,\dots,r$}
          \STATE Using Remark~\ref{remark:exp_lattice_fields}, compute the exponent lattice
                 $M_i \subseteq \mathbb{Z}^k$ of $(f_1, \dots, f_k)$ in the $\mathbb{F}_p$-algebra
                 $\mathbb{F}_p \otimes_{\mathbb{Z}} R$.
          \FOR {$j=1,\dots,e_i-1$}
             \STATE Assume that $M_{i}$ is generated by $\{ b_1, \dots, b_m \} \subseteq \mathbb{Z}^k$.
             \STATE For $s = 1, \dots, m$ form the elements
                    $h_r = (f_1^{b_{s 1}} \cdots f_k^{b_{s k}}-1)/p_i^j \in R$.
              \STATE Compute the solution space $M' \subseteq \mathbb{Z}^m$ of the linear equation
                     over~$\mathbb{Z}$ in the indeterminates $y_1, \dots, y_m$ given by
                     $$\bar{h}_1 y_1 + \cdots + \bar{h}_m y_m = 0 \quad \text{ in } P/\langle I, p_i \rangle.$$
              \STATE Replace $M_i$ by the lattice
                     $\{c_1b_1+\cdots + c_m b_m \mid (c_1,\dots,c_m) \in M'\} \subseteq \mathbb{Z}^k$.
          \ENDFOR
      \ENDFOR
      \RETURN the lattice $M_1 \cap \cdots \cap M_r$
  \ENDIF
\end{algorithmic}
This is an algorithm which computes the exponent lattice of $(f_1, \dots, f_k)$ in~$R$.
\end{algo}

\begin{proof}
A tuple $a = (a_1, \dots, a_k) \in \mathbb{Z}^k$ is in the exponent lattice of $(f_1, \dots, f_k)$ if and
only if $f_1^{a_1} \cdots f_k^{a_k}-1 \in I$. By Proposition~\ref{prop:gb_lcm} we have
$I =(I : \langle N \rangle) \cap \langle I, N \rangle$. In the case $I \cap \mathbb{Z} = \langle 0 \rangle$,
we have $I : \langle N \rangle = I \mathbb{Q}[x_1, \dots, x_n] \cap P$ by the same proposition. Line~3 therefore
yields the exponent lattice of $(f_1, \dots, f_k)$ in $P/(I : \langle N \rangle)$. It remains to prove that
lines 12--24 determine the exponent lattice of $(f_1, \dots, f_k)$ in $P/\langle I, N \rangle$. Since we have
$\langle I, N \rangle = \bigcap_{i=1}^r \langle I, p_i^{e_i} \rangle$, it is enough to show that
lines 15--21 compute the exponent lattice of $(f_1, \dots, f_k)$ in $P/\langle I, p_i^{e_i} \rangle$.
Line~15 yields the exponent lattice of $(f_1, \dots, f_k)$ in $P/\langle I, p_i \rangle$. It then follows
from Proposition~\ref{prop:hensel_lifting} that the $j$-th iteration of the for loop in lines 16--21 correctly
computes the exponent lattice of $(f_1, \dots, f_k)$ in $P/\langle I, p_i^{j+1} \rangle$.
\end{proof}

Let us collect some remarks about the implementation of this algorithm.

\begin{rem}\label{rem:impl_exp_lattice}
Suppose we are in the setting of Algorithm~\ref{alg:exp_lattice_z-algebra}.
\begin{enumerate}
\item[(a)] A non-zero generator $q$ of the ideal $I \cap \mathbb{Z}$ in line~1 is given by the unique integer
contained in a reduced strong Gr\"obner basis of~$I$. If $R$ is a finite $\mathbb{Z}$-algebra, then
$q$ is zero if and only if the rank of $R^+$ is non-zero. The rank of an explicitly given finite
$\mathbb{Z}$-algebra can be determined in polynomial time using a Smith normal form computation.

\item[(b)] As illustrated in Example~\ref{example:gb_lcm} there can be a proper divisor $M$ of $N$
satisfying $I : \langle M \rangle = I : \langle M \rangle \cap \langle I, M\rangle$. By
determining the smallest number with this property, unnecessary iterations in the else-branch of this
algorithm can be avoided. If $R$ is a finite $\mathbb{Z}$-algebra, then the smallest number with this
property is given by the torsion exponent of $R$. It can be determined in polynomial time using a
Smith normal form computation if $R$ is explicitly given.

\item[(c)] If $R$ is a finite $\mathbb{Z}$-algebra, then the $\mathbb{Q}$-algebra in line~3 and
the $\mathbb{F}_p$-algebra in line~15 are zero-dimensional. Exponent lattices in an explicitly given
zero-dimensional $\mathbb{Q}$-algebra can be computed in polynomial time
(see Algorithm~8.3 in~\cite{lenstra2018algorithms}). For zero-dimensional $\mathbb{F}_p$-algebras, the
problem can be reduced to the discrete logarithm problem in finite fields
(see Algorithm~3.20 in~\cite{kreuzer2024binomial}).

\item[(d)] In line~19 we need to compute the solution space of the linear equation over $\mathbb{Z}$ in the
indeterminates $y_1, \dots, y_m$ given by $$g_1 y_1 + \cdots + g_m y_m = 0$$ in
$P/\langle I, p_i \rangle$. This can be achieved by checking for all
$(a_1, \dots, a_m) \in \mathbb{Z}^m$ with $0 \leq a_\ell \leq p_i$ for $\ell = 1, \dots, m$
whether $a_1g_1 + \cdots + a_m g_m \in \langle I, p_i \rangle$. Alternatively, one can perform a
syzygy calculation using Gr\"obner basis techniques. In particular, for large $p_i$, this might be more
efficient. If $R$ is an explicitly given finite $\mathbb{Z}$-algebra, we can 
use~\cite[Prop.~2.6]{kreuzer2024efficient} to solve this linear equation efficiently.

\end{enumerate}
\end{rem}

The following example illustrates how Algorithm~\ref{alg:exp_lattice_z-algebra} can be applied to compute
exponent lattices in finite $\mathbb{Z}$-algebras.

\begin{exa}
Let $I = \left\langle x^2+x+1,\, y^2+y+1,\, 6z^2,\, z^3 \right\rangle$, and consider the finite
$\mathbb{Z}$-algebra $R=\mathbb{Z}[x,y,z]/I$. Let $f_1 = -xyz - xz + 1$, $f_2=y + 1$ and $f_3=xy + x + y + 1$.
We apply Algorithm~\ref{alg:exp_lattice_z-algebra} to compute the exponent lattice of
$\left(\bar{f}_1, \bar{f}_2, \bar{f}_3\right)$ in~$R$.
\begin{enumerate}
\item[1:] Since $R$ is a finite $\mathbb{Z}$-algebra, we compute its rank given by 8.

\item[2:] The rank of $R$ is non-zero, which implies $I \cap \mathbb{Z} = \langle 0 \rangle$.

\item[3:] Using Remark~\ref{remark:exp_lattice_fields}, we compute the exponent lattice
                $\Lambda = \langle (0,6,0), (0,0,3) \rangle$ in the zero-dimensional $\mathbb{Q}$-algebra
                $\mathbb{Q} \otimes_{\mathbb{Z}} R$.

\item[4,5:] Since $R$ is a finite $\mathbb{Z}$-algebra, we compute its torsion exponent given by 6.

\item[9:] We recursively apply the Algorithm to compute the exponent lattice of
                $\left(\bar{f}_1, \bar{f}_2, \bar{f}_3\right)$ in~$P/\langle I, 6 \rangle$.

\item[1:] We have $q = 6$.

\item[13:] We determine the factorization $6 = 2 \cdot 3$.

\item[15:] Using Remark~\ref{remark:exp_lattice_fields}, we compute generators $(4,0,0)$, $(0,3,0)$ and
                $(0,0,3)$ of the exponent lattice $M_1$ of $\left(\bar{f}_1, \bar{f}_2, \bar{f}_3\right)$ modulo
                $I\mathbb{F}_2[x,y,z]$ and generators $(3,0,0)$, $(0,6,0)$, and $(0,0,3)$ of the exponent lattice $M_2$
                modulo $I\mathbb{F}_3[x,y,z]$.
\item[20:] The exponent lattice of $\left(\bar{f}_1, \bar{f}_2, \bar{f}_3\right)$ in $R$ is given by
                $\Lambda \cap M_1 \cap M_2 = \langle (0,6,0), (0,0,3) \rangle$.

\end{enumerate}
\end{exa}

It is an open question to what extent Algorithm~\ref{alg:exp_lattice_z-algebra} can be generalized to the
non-commutative case. A straightforward generalization does not seem to be possible, since in this case the
multiplicative relations between units are in general not computable. This follows from the fact that the
subgroup membership problem is undecidable for $4 \times 4$ integral matrices (see~\cite{mihailova1968occurrence}).
Also note that multiplicative relations in general do not form a lattice in the non-commutative case.

\bigskip\bigbreak
%%%%%%%%%%%%%%%%%%%%%%%%%%%%%%%%%%%%%%%%%%%%%%%%%%%%%%%%%%%%%%%%%%%%%%
%
%  Section 4: Reduced Finite ZZ-Algebras
%
%%%%%%%%%%%%%%%%%%%%%%%%%%%%%%%%%%%%%%%%%%%%%%%%%%%%%%%%%%%%%%%%%%%%%%

\section{The Unit Group of Reduced Finite \texorpdfstring{$\mathbb{Z}$}{Z}-Algebras}

Let us begin by considering the case of integral finite $\mathbb{Z}$-algebras, i.e, algebras of the form
$P/\mathfrak{p}$ where~$\mathfrak{p}$ is a prime ideal in~$P$.

\begin{rem}{\bf (Computing the Unit Group of Integral Finite $\mathbb{Z}$-Algebras)}\label{remark:integral}\\
Let $\mathfrak{p}$ be a prime ideal in $P$ such that $R = P/\mathfrak{p}$ is a finite $\mathbb{Z}$-algebra.
\begin{enumerate}
\item[(a)] If $\mathfrak{p} \cap \mathbb{Z} = \langle 0 \rangle$, then $K = \mathbb{Q} \otimes_{\mathbb{Z}} R$
is a number field. Since $P/\mathfrak{p}$ is integral over $\mathbb{Z}$ and its rank equals
$\dim_{\mathbb{Q}}(K)$, the ring $R$ is an order in~$K$. Generators of the unit group of $R$ can
therefore be computed using for example the algorithms given in~\cite{cohen1997subexponential}
or~\cite{biasse2014subexponential}. These algorithms require that the field $K$ is given by a primitive
element. Such an element can be determined using one of the methods described
in~\cite{yokoyama1989computing} or Algorithm~6.3 in~\cite{lenstra2018algorithms}.

\item[(b)] If $\mathfrak{p} \cap \mathbb{Z} = \langle p \rangle$ for a prime number $p$, then $P/\mathfrak{p}$
is isomorphic to the finite field $K = \mathbb{F}_p \otimes_{\mathbb{Z}} R$. The problem therefore
reduces to computing a primitive root of~$K$. Algorithms for this task can be found
in~\cite{shoup1990searching} or in~\cite{dubrois2006efficient}.

\end{enumerate}    
\end{rem}

Let us now consider the case of a reduced finite $\mathbb{Z}$-algebra $R = P/I$. If
$I \cap \mathbb{Z} = \langle n \rangle$ for some $n \in \mathbb{Z} \setminus \{0\}$, then the minimal prime
ideals of~$R$ are maximal ideals and therefore pairwise coprime. By the
Chinese Remainder Theorem computing the unit group then reduces to the case discussed above. If
$I \cap \mathbb{Z} = \langle 0 \rangle$, then the minimal prime ideals of~$R$ need not be pairwise coprime.

\begin{exa}
Consider the ideal $I = \langle x^2+x+1, y^2+y+1 \rangle$. Its minimal prime ideals are given by
$\mathfrak{p}_1 = \langle x - y, y^2 + y + 1 \rangle$ and
$\mathfrak{p}_2 = \langle x + y + 1, y^2 + y + 1 \rangle$ and we have
$\mathfrak{p}_1 + \mathfrak{p}_2 = \langle x + 2, y + 2, 3 \rangle$.
\end{exa}

This example demonstrates that we cannot directly reduce to the integral case. Instead we notice, that if $R$ is
torsion-free, then $R$ is an order in the reduced zero-dimensional $\mathbb{Q}$-algebra
$\mathbb{Q} \otimes_{\mathbb{Z}} R$.

\begin{defi}
Let $A$ be a zero-dimensional reduced $\mathbb{Q}$-algebra. A subring $\mathcal{O}$ of $A$ is called an
\textbf{order} if there is a basis $a_1, \dots, a_m$ of $A$ such that
$\mathcal{O} = \mathbb{Z}a_1 + \cdots + \mathbb{Z}a_m$.
\end{defi}

The unit group of an order in a zero-dimensional reduced $\mathbb{Q}$-algebra, can be computed using the algorithm
presented in Section~3 of \cite{faccin2013computing}. In the following we present a modified version
of this algorithm.

\begin{lem}\label{lemma:order}
Let $A$ be a reduced zero-dimensional $\mathbb{Q}$-algebra and let $e_1, e_2$ be orthogonal idempotents with
$e_1+e_2 = 1$. Let $\mathcal{O}$ be an order in $A$. Consider the ideal
$J = (e_1 \mathcal{O} \cap \mathcal{O}) + (e_2 \mathcal{O} \cap \mathcal{O})$ and form the ring
$S = \mathcal{O}/J$.
Consider the ring homomorphisms $\varphi_i : e_i \mathcal{O} \rightarrow S$ given by $\varphi_i(e_i a) = a+J$
for $a \in \mathcal{O}$. Then we have
$$  
\mathcal{O} = \{a_1 + a_2 \mid a_i \in e_i \mathcal{O} \text{ and } \varphi_1(a_1) = \varphi_2(a_2) \}
$$
\end{lem}
    
\begin{proof}
See \cite{faccin2013computing}, Lemma~3.1.
\end{proof}

Given an order $\mathcal{O}$ in a reduced zero-dimensional $\mathbb{Q}$-algebra $A$ and primitive idempotents
$e_1, \dots, e_k$ of $A$, we can compute generators of $(e_i \mathcal{O})^\times$ since $e_i \mathcal{O}$ is an
order in the number field $e_i A$. We can then iteratively apply Lemma~\ref{lemma:order} to determine
$\mathcal{O}^\times$.

\begin{algo}{\bf (Computing the Unit Group of an Order)}\label{alg:unit_group_order}\ \\
Let $A$ be a reduced zero-dimensional $\mathbb{Q}$-algebra, and let $\mathcal{O}$ be an order in~$A$.
Consider the following sequence of instructions.

\begin{algorithmic}[1]
   \STATE Compute the primitive idempotents $e_1, \dots, e_m$ of $A$.
   \STATE Compute $U = (e_1 \mathcal{O})^\times$.
   \FOR{$j=2, \dots, m$}
      \STATE Set $f = e_1 + \cdots + e_{j-1}$.
      \STATE Compute generators of $f \mathcal{O} \cap \mathcal{O}$ and $e_j \mathcal{O} \cap \mathcal{O}$ and
          form the ideal $J = (f \mathcal{O} \cap \mathcal{O}) + (e_j \mathcal{O} \cap \mathcal{O})$ in~$\mathcal{O}$.
      \STATE Compute generators $e_j h_1, \dots, e_j h_\ell$ of $(e_j \mathcal{O})^\times$.
      \STATE Assume that $U = \{f g_1, \dots, f g_k\}$.
      \STATE Compute a set of generators $B \subseteq \mathbb{Z}^{k+\ell}$ of the exponent lattice $\Lambda$
          of the tuple $\left(g_1, \dots, g_k, h_1^{-1}, \dots, h_\ell^{-1}\right)$ in $\mathcal{O}/J$.
      \STATE Set $U = \left\{ fg_1^{b_1}\cdots g_k^{b_k} + e_j h_1^{b_{k+1}} \cdots h_\ell^{b_{k+\ell}}
          \mid b \in B \right\}$.
   \ENDFOR
\RETURN $U$.
\end{algorithmic}
This is an algorithm which computes a set of generators of~$\mathcal{O}^\times$.
\end{algo}

\begin{proof}
It suffices to show that after the $j$-th iteration $U$ generates the unit group of the order
$(e_1+ \cdots + e_j)\mathcal{O}$ in $(e_1+ \cdots + e_j)A$.
Let $r \in \mathcal{O}$ and consider the group homomorphisms $\varphi_f : f \mathcal{O} \rightarrow
\mathcal{O}/J$ given by $\varphi_f(fr) = r + J$ and $\varphi_{e_j} : e_j \mathcal{O} \rightarrow \mathcal{O}/J$
given by $\varphi_{e_j}(e_jr) \mapsto r + J$. Let $g = f g_1^{c_1} \cdots g_k^{c_k}$ in~$U$ and
$h = e_j h_1^{c_{k+1}} \cdots h_\ell^{c_{k+\ell}}$ in~$(e_j \mathcal{O})^\times$. Then by
Lemma~\ref{lemma:order} the element $g-h$ is in the unit group of $(e_1+ \cdots + e_j)\mathcal{O}$ if and
only if
$$
 \varphi_f(g) = g_1^{c_1} \cdots g_k^{c_k}+J = h_1^{c_{k+1}} \cdots h_\ell^{c_{k+\ell}}+J = \varphi_{e_j}(e_j h).
$$
This is equivalent to $(c_1, \dots, c_{k+\ell}) \in \Lambda$.
\end{proof}

Let $R=P/I$ be a finite $\mathbb{Z}$-algebra. If $R$ is torsion-free and reduced, then it is an order in the
zero-dimensional $\mathbb{Q}$-algebra $A = \mathbb{Q} \otimes_{\mathbb{Z}} R$. Consequently,
Algorithm~\ref{alg:unit_group_order} can be applied to compute the unit group of the order~$R$ in~$A$. In this case
the steps of this algorithm can be performed as follows.

\begin{lem}\label{lemma:torsion_free}
Let $R=P/I$ be a reduced torsion-free finite $\mathbb{Z}$-algebra, and let
$\mathfrak{p}_1, \dots, \mathfrak{p}_m$ be the minimal prime ideals of~$I$.
\begin{enumerate}
\item[(a)] The $\mathbb{Q}$-algebra $A = \mathbb{Q} \otimes_{\mathbb{Z}} R$ is zero-dimensional and its
maximal ideals are given by $\mathfrak{m}_i = \mathfrak{p}_i \mathbb{Q}[x_1, \dots, x_n]$ for
$i=1, \dots, m$. We can therefore compute elements $q_i \in \bigcap_{i\neq j} \mathfrak{m}_j$ and
$p_i \in \mathfrak{m}_i$ such that $q_i+p_i = 1$. The residue classes
$\bar{q}_1, \dots, \bar{q}_m$ in $A$ then form the primitive idempotents of $A$

\item[(b)] The ideal $\bar{q}_i R$ in $A$ is isomorphic to $\bar{q}_i (P/\mathfrak{p}_i)$.

\item[(c)] We have $\bar{q}_i R \cap R = \bigcap_{j \neq i} \mathfrak{p}_i/I$.

\item[(d)] For $f = \sum_{i \neq j} q_i$ we have $\bar{f} R \cap R = \mathfrak{p}_j/I$.

\end{enumerate}
\end{lem}

\begin{proof}
Part~(a) is a direct consequence of the Chinese Remainder Theorem
(see Lemma~3.7.4 in~\cite{kreuzer2000computational}), and (b) follows from the fact
that $\bar{q}_iA$ is isomorphic to $A/\mathfrak{m}_i$.

To proof~(c), we notice that the ideal $\bar{q}_i R \cap R$ is contained in the right-hand side, since
$q_i \in \bigcap_{i\neq j} \mathfrak{m}_j$. To show the opposite inclusion,
let $f \in P$ such that $\bar{f} \in \bigcap_{i\neq j} \mathfrak{p}_j/I$. Then we have $f = q_if+p_if$.
Since $p_i \in \mathfrak{m}_i$, this implies $p_if \in \bigcap_{i=1,\dots,m} \mathfrak{p}_i = I$. Hence,
$\bar{f} = \bar{q}_i\bar{f} \in \bar{q}_i R \cap R$. Part~(d) follows analogously.
\end{proof}

Using these observations, we can adapt Algorithm~\ref{alg:unit_group_order} as follows.

\begin{cor}{\bf (Computing the Unit Group of a Reduced Torsion-Free Finite 
$\mathbb{Z}$-Algebra)}\label{alg:unit_group_torsion_free}\ \\
Let $R=P/I$ be a reduced torsion-free finite $\mathbb{Z}$-algebra. Consider the following sequence of instructions.

\begin{algorithmic}[1]
    \STATE Compute the minimal prime ideals $\mathfrak{p}_1, \dots, \mathfrak{p}_m$ of $I$.
    \FOR{$i=1, \dots, m$}
    \STATE Compute elements $q_i \in \bigcap_{i\neq j} \mathfrak{p}_j \mathbb{Q}[x_1, \dots, x_n]$ and
        $p_i \in \mathfrak{p}_i\mathbb{Q}[x_1, \dots, x_n]$ such that $q_i+p_i = 1$.
    \ENDFOR
    \STATE Using Remark~\ref{remark:integral}, compute a set of polynomials~$U$ such that their residue classes
       generate $(P/\mathfrak{p}_1)^\times$.
    \FOR{$j=2, \dots, m$}
            \STATE Form the ideal $J = \bigcap_{1 \leq i \leq j-1} \mathfrak{p}_i + \mathfrak{p}_j$.
                \STATE Using Remark~\ref{remark:integral}, compute polynomials $h_1, \dots, h_\ell$ such that their
                    residue classes generate $(P/\mathfrak{p}_j)^\times$.
                \STATE Assume that $U = \{g_1, \dots, g_k\}$, and compute a set of generators
                    $B \subseteq \mathbb{Z}^{k+\ell}$ of the exponent lattice of
                    $\left(\bar{g}_1, \dots, \bar{g}_k, \bar{h}_1^{-1}, \dots, \bar{h}_\ell^{-1}\right)$ in~$P/J$.
                \STATE Set $U = \left\{fg_1^{b_1}\cdots g_k^{b_k} + q_j h_1^{b_{k+1}} \cdots h_\ell^{b_{k+\ell}}
                    \mid b \in B \right\}$ where $f = q_1 + \cdots + q_{j-1}$.
            \ENDFOR
        \RETURN $U$.
\end{algorithmic}
This is an algorithm which computes a set of polynomials such that their residue classes generate~$R^\times$.
\end{cor}

\begin{proof}
Since $R$ is torsion-free and reduced, it is an order in the zero-dimensional $\mathbb{Q}$-algebra
$\mathbb{Q} \otimes_{\mathbb{Z}} R$. Now we show that the steps of this algorithm correspond to the steps
of Algorithm~\ref{alg:unit_group_order}.

Assume that at the start of the $j$-th iteration of the for-loop in lines 6--11 the residue classes of the
elements in~$U$ generate the unit group of $R_{j-1} = P/(\mathfrak{p}_1 \cap \cdots \cap \mathfrak{p}_{j-1})$.
Consider the finite $\mathbb{Z}$-algebra $R_j = P/(\mathfrak{p}_1 \cap \cdots \cap \mathfrak{p}_j)$. The
residue classes of the elements $f = q_1 + \cdots + q_{j-1}$ and~$q_j$ form orthogonal idempotents in
$\mathbb{Q} \otimes_{\mathbb{Z}} R_j$ with $\bar{f}+\bar{q}_j = 1$. By Lemma~\ref{lemma:torsion_free}.b we then
have $\bar{q}_j R_j \cong \bar{q}_j (P/\mathfrak{p}_j)$ and $\bar{f} R_j \cong \bar{f} R_{j-1}$. Furthermore
Lemma~\ref{lemma:torsion_free}.c yields $\bar{q}_j R_j \cap R_j = (\mathfrak{p}_1 \cap \cdots \cap \mathfrak{p}_{j-1})/I$
and $\bar{f} R_j \cap R_j = \mathfrak{p}_j/I$. This shows that the ideal~$J$ in line~7 corresponds to the
ideal~$J$ in line~5 of Algorithm~\ref{alg:unit_group_order}.
\end{proof}

Let us apply this algorithm to a concrete example.

\begin{exa}
Let $I = \left\langle x^2+x+1, y^2+y+1, z^2+z+1 \right\rangle$, and consider the finite $\mathbb{Z}$-algebra
$R = \mathbb{Z}[x,y,z]/I$. We follow the steps of Corollary~\ref{alg:unit_group_torsion_free} to compute
generators of~$R^\times$.
\begin{enumerate}
\item[1:] We compute the minimal prime ideals of $I$ and obtain
  \begin{align*}
     \mathfrak{p}_1 &= \left\langle y - z, x + z + 1, z^2 + z + 1 \right\rangle \\
     \mathfrak{p}_2 &= \left\langle y - z, x - z, z^2 + z + 1 \right\rangle \\
     \mathfrak{p}_3 &= \left\langle y + z + 1, x + z + 1, z^2 + z + 1 \right\rangle \\
     \mathfrak{p}_4 &= \left\langle y + z + 1, x - z, z^2 + z + 1 \right\rangle
  \end{align*}

\item[2--4:] We compute the primitive idempotents
                \begin{align*}
                    e_1 &= 1/3\bar{x}\bar{y} + 1/3\bar{x}\bar{z} - 1/3\bar{y}\bar{z} + 1/3\bar{x} + 1/3 \\
                    e_2 &= -1/3\bar{x}\bar{y} - 1/3\bar{x}\bar{z} - 1/3\bar{y}\bar{z} - 1/3\bar{x} - 1/3\bar{y} - 1/3\bar{z}\\
                    e_3 &= -1/3\bar{x}\bar{y} + 1/3\bar{x}\bar{z} + 1/3\bar{y}\bar{z} + 1/3\bar{z} + 1/3 \\
                    e_4 &= 1/3\bar{x}\bar{y} - 1/3\bar{x}\bar{z} + 1/3\bar{y}\bar{z} + 1/3\bar{y} + 1/3
                \end{align*}
                of the zero-dimensional $\mathbb{Q}$-algebra $\mathbb{Q} \otimes_{\mathbb{Z}} R$.

\item[5:] Using Remark~\ref{remark:integral}, we determine a set of generators $U= \{ \bar{z}+1 \}$ of
                $(R/\mathfrak{p}_1)^\times$.

\item[7:] Form the ideal $J = \mathfrak{p}_1+ \mathfrak{p}_2 = \langle x+1, y+2, z+2, 3 \rangle$.

\item[8:] Using Remark~\ref{remark:integral}, we
                compute $(R/\mathfrak{p}_2)^\times = \langle \bar{z}+1 \rangle$.

\item[9:] Using Remark~\ref{remark:exp_lattice_fields}, we compute generators $(1,1)$ and $(0,2)$ of the
                exponent lattice of $(\bar{z}+1, (\bar{z}+1)^{-1})$ in the finite field $R/J$.

\item[10:] Compute $e_1 (\bar{z}+1) + e_2 (\bar{z}+1) = \bar{z}+1$ and
                $e_1+e_2(\bar{y}+1)^2 = -\bar{y}\bar{z}-\bar{y}$ and set $U = \{\bar{z}+1, -\bar{y}\bar{z}-\bar{y}\}$.

\item[7:] Form the ideal $J = (\mathfrak{p}_1 \cap \mathfrak{p}_2)+\mathfrak{p}_3 = \langle x+2, y+2, z+2, 3 \rangle$.

\item[8:] Using Remark~\ref{remark:integral}, we
                compute $(R/\mathfrak{p}_2)^\times = \langle \bar{z}+1 \rangle$.

\item[9:] Using Remark~\ref{remark:exp_lattice_fields}, we compute generators $(1,0,1)$, $(0,1,0)$ and
                $(0,0,2)$ of the exponent lattice of $(\bar{z}+1, -\bar{y}\bar{z}-\bar{y}, (\bar{z}+1)^{-1})$ in~$R/J$.

\item[10:] Set $U = \{\bar{z}+1, -\bar{y}\bar{z}-\bar{y}, \bar{x}\bar{z}+\bar{x}+\bar{z}+1\}$.

\item[7:] Form the ideal $J = (\mathfrak{p}_1 \cap \mathfrak{p}_2 \cap \mathfrak{p}_3)+\mathfrak{p}_4 =
                \langle z^2 + z + 1, x + 2z, y + z + 1, 3\rangle$.

\item[8:] Using Remark~\ref{remark:integral}, we
                compute $(R/\mathfrak{p}_2)^\times = \langle \bar{z}+1 \rangle$.

\item[9:] Using Remark~\ref{remark:exp_lattice_fields}, we compute generators $(1,0,0,1)$, $(0,1,0,2)$,
                $(0,0,1,2)$ and $(0,0,0,6)$ of the exponent lattice in~$R/J$ of
                $$
                \left(\bar{z}+1, -\bar{y}\bar{z}-\bar{y}, \bar{x}\bar{z}+\bar{x}+\bar{z}+1,(\bar{z}+1)^{-1}\right).
                $$

\item[10:] Set $U = \{\bar{z}+1, -\bar{y}\bar{z}-\bar{y}, \bar{x}\bar{z}+\bar{x}+\bar{z}+1, 1\}$.

\item[12:] The algorithm returns the generators
                $\bar{z}+1, -\bar{y}\bar{z}-\bar{y}, \bar{x}\bar{z}+\bar{x}+\bar{z}+1$ of~$R^\times$.

\end{enumerate}
\end{exa}

A reduced finite $\mathbb{Z}$-algebra need not be torsion-free, but we can decompose it into a direct product of
finitely many finite fields and a torsion-free algebra.

\begin{lem}\label{lemma:reduction_to_torsion_free}
Let $R$ be a reduced finite $\mathbb{Z}$-algebra. Let $\mathfrak{p}_1, \dots, \mathfrak{p}_r$ be the minimal
prime ideals of~$R$ of height $n$, let $\mathfrak{m}_1, \dots, \mathfrak{m}_{s}$ be the maximal
ideals of~$R$, and let $J = \mathfrak{p}_1 \cap \cdots \cap \mathfrak{p}_r$. Then $R/J$ is torsion-free and we
have $R \cong R/J \times R/\mathfrak{m}_1 \times \cdots \times R/\mathfrak{m}_s$.
\end{lem}

\begin{proof}
As a reduced ring, $R$ does not have embedded prime ideals. Therefore $J$ is not contained in any
$\mathfrak{m}_i$. Since $\mathfrak{m}_1, \dots, \mathfrak{m}_s$ are maximal ideals the isomorphism follows
directly by the Chinese Remainder Theorem. To prove that $R/J$ is torsion-free, we note that the prime ideals
$\mathfrak{p}_i$ satisfy $\mathfrak{p}_i \mathbb{Q}[x_1, \dots, x_n] \cap P$. Now let $f \in P$ and assume
there exists $k \in \mathbb{Z} \setminus \{0\}$ with $kf \in J$. Then we have $kf \in \mathfrak{p}_i$ and by
the above observation $f \in \mathfrak{p}_i$ for all $i=1, \dots, k$. This shows $f \in J$.
\end{proof}

Thus we can now combine our results and obtain the following algorithm.

\begin{algo}{\bf (Computing the Unit Group of a Reduced Finite $\mathbb{Z}$-Algebra)}\label{alg:units_reduced}
Let $R = P/I$ be a reduced finite $\mathbb{Z}$-algebra. Consider the following sequence of instructions. 

\begin{algorithmic}[1]
            \STATE Let $G =  [\;]$.
            \STATE Compute the prime decomposition $I = \mathfrak{p}_1 \cap \cdots \cap \mathfrak{p}_r \cap
                \mathfrak{m}_1 \cap \cdots \cap \mathfrak{m}_s$ where $\mathfrak{p}_1, \dots, \mathfrak{p}_r$ are
                prime ideals of height $n$ and $\mathfrak{m}_1, \dots, \mathfrak{m}_s$ are maximal ideals.
            \STATE Compute $J = \mathfrak{p}_1 \cap \cdots \cap \mathfrak{p}_r$.
            \STATE Apply Algorithm~\ref{alg:unit_group_order} to compute a set of polynomials $H$ such that their
                residue classes generate the unit group of the order $P/J$ in
                $\mathbb{Q}[x_1, \dots, x_n]/J \mathbb{Q}[x_1, \dots, x_n]$.

            \STATE Using the Chinese Remainder Theorem, compute $e_1, \dots, e_{s+1} \in P$ such that their residue
                classes form orthogonal idempotents of $R$ with $R \cong \bar{e}_1 R \times \cdots \times \bar{e}_{s+1} R$
                and $\bar{e}_i R \cong P/\mathfrak{m}_i$ for $i=1, \dots, s$ and $\bar{e}_{s+1}R \cong P/J$.
            \FORALL{$h \in H$}
            \STATE Add $\bar{e}_{s+1} \bar{h} + \sum_{i=1}^{s} \bar{e}_i$ to~$G$.
            \ENDFOR
            \FOR{$i=1, \dots,s$}
                \STATE Using Remark~\ref{remark:integral}, compute $g_i \in P$ such that $\bar{g}_i$ generates
                    $(P/\mathfrak{m}_i)^\times$.
                \STATE Add $\bar{e}_i\bar{g}_i + \sum_{j \neq i} \bar{e}_j$ to $G$.
            \ENDFOR
            \RETURN $G$.
\end{algorithmic}
This is an algorithm which computes a generating set of the unit group $R^\times$.
\end{algo}

\begin{proof}
By Lemma~\ref{lemma:reduction_to_torsion_free}, we have
$R^\times \cong (P/J)^\times \times (P/\mathfrak{m}_1)^\times \times \cdots \times (P/\mathfrak{m}_s)^\times$,
and $P/J$ is torsion-free. This shows that the units in $R$ are given by the residue classes of elements of the
form $e_1f_1+\cdots +e_s f_s+e_{s+1}g$ where $\bar{f}_i$ is a unit in $(P/\mathfrak{m}_i)$ and $\bar{h}$ is a
unit in~$(P/J)$. We therefore conclude that the elements computed in line~7 and in line~11 generate~$R^\times$.
\end{proof}

As noted in Remark~\ref{remark:complexity}, the computations in lines~2,3 and~5 can be performed efficiently
if $R$ is explicitly given.

\begin{exa}
Consider the reduced finite $\mathbb{Z}$-algebra $R =\mathbb{Z}[x,y,z]/I$ where
$$
I = \left\langle 3x, \, xz - x, \, y^2 + z, \, x^2 + xy, \, z^3 - 1 \right\rangle.
$$
Let us apply Algorithm~\ref{alg:units_reduced} to compute generators of $R^\times$.
\begin{enumerate}
\item[1:] Set $G = [\;]$.

\item[2:] We compute the minimal prime ideals of~$I$ and obtain
                $\mathfrak{p}_1 = \langle z - 1, x, y^2 + 1 \rangle$,
                $\mathfrak{p}_2 = \langle x, z^2 + z + 1, y^2 + z \rangle$ and
                $\mathfrak{m} = \langle 3, z - 1, x + y, y^2 + 1 \rangle$.

\item[3:] Compute $J = \mathfrak{p}_1 \cap \mathfrak{p}_2$.

\item[4:] Using Algorithm~\ref{alg:unit_group_order}, we obtain the following generators of $(P/J)^\times$.
                \begin{align*}
                    h_1 &= 9\bar{y}\bar{z}^2 - 17\bar{y}\bar{z} - 15\bar{z}^2 + 9\bar{y} + 15\bar{z} \\
                    h_2 &= 15\bar{y}\bar{z} + 9\bar{z}^2 - 15\bar{y} - 17\bar{z} + 9 \\
                    h_3 &= -56\bar{y}\bar{z} - 32\bar{z}^2 + 56\bar{y} + 65\bar{z} - 32
                \end{align*}

\item[5:] Determine orthogonal idempotents $e_1 = \bar{x}\bar{y} + \bar{y}^2 + \bar{z}$ and
                $e_2 = -\bar{x}\bar{y} + 1$ of~$R$ such that $e_1+e_2 = 1$, $e_1 R \cong P/\mathfrak{m}$ and
                $e_2 R \cong P/J$.

\item[7:] Add $e_1+e_2h_1$, $e_1+e_2h_2$ and $e_1+e_2h_3$ to $G$.

\item[10:] Using Remark~\ref{remark:integral}, we compute a generator $g = \bar{y}+1$ of $(P/\mathfrak{m})^\times$.

\item[11:] Add $e_1g+e_2$ to $G$.

\item[13:] The generators of $R^\times$ are given by
                \begin{align*}
                    e_1+e_2h_1 &= 9\bar{y}\bar{z}^2 + \bar{x}\bar{y} - 17\bar{y}\bar{z} - 15\bar{z}^2 + \bar{x} + 9\bar{y} + 15\bar{z} \\
                    e_1+e_2h_2 &= 15\bar{y}\bar{z} + 9\bar{z}^2 - 15\bar{y} - 17\bar{z} + 9\\
                    e_1+e_2h_3 &= -56\bar{y}\bar{z} - 32\bar{z}^2 + 56\bar{y} + 65\bar{z} - 32 \\
                    e_1g+e_2 &= -\bar{x}+1.
                \end{align*}

\end{enumerate}
\end{exa}

\bigskip\bigbreak
%%%%%%%%%%%%%%%%%%%%%%%%%%%%%%%%%%%%%%%%%%%%%%%%%%%%%%
%
%  Section 5: Reduction to the Reduced Case
%
%%%%%%%%%%%%%%%%%%%%%%%%%%%%%%%%%%%%%%%%%%%%%%%%%%%%%%

\section{The Unit Group of Non-Reduced Finite \texorpdfstring{$\Z$}{Z}-Algebras}

Using the results of the previous subsection, in particular Algorithm~\ref{alg:units_reduced}, we can now compute
generators of the unit group of $R/\Rad(0)$. The next task is to lift these generators to a generating set of $R^\times$.

\begin{lem}\label{lemma:canonical_epi}
Let $R$ be a finite $\mathbb{Z}$-algebra.
\begin{enumerate}
\item[(a)] The canonical homomorphism $\varphi \colon R \rightarrow R/\Rad(0)$ induces a surjective group
homomorphism $\varphi^\times \colon R^\times \rightarrow (R/\Rad(0))^\times$.

\item[(b)] The kernel of $\varphi^\times$ is given by $1+ \Rad(0)$.

\end{enumerate}
\end{lem}

\begin{proof}
 See \cite{malek1972group}, Lemma 1.1.5.
\end{proof}

This lemma shows $R^\times/(1+\Rad(0)) \cong (R/\Rad(0))^\times$. Therefore, if we have generators of
$(R/\Rad(0))^\times$ and generators of $1+\Rad(0)$, we obtain generators of~$R^\times$ using the following
well-known result.

\begin{rem}
Let $M$ be a $\mathbb{Z}$-module and $N \subseteq M$ a submodule. Assume that the residue classes
of $m_1, \dots, m_k \in M$ generate $M/N$, and let $N$ be generated by $n_1, \dots, n_\ell$. Then
$m_1, \dots, m_k, n_1, \dots, n_\ell$ generate $M$.
\end{rem}

Since we already saw how to determine generators of $R/\Rad(0)^\times$, it remains to compute generators
of $1+\Rad(0)$. If $\Rad(0)^2 = 0$, then elements $1+f, 1+g$ of~$1+ \Rad(0)$ satisfy $(1+f)(1+g) \equiv 1+f+g$.
In this case generators of the additive group of $\Rad(0)$, immediately yield generators of $1+\Rad(0)$.
If the nilpotency index of $\Rad(0)$ is greater than 2, we can inductively compute generators of $1+ \Rad(0)^{i-1}$
in~$R/\Rad(0)^i$.

\begin{algo}{\bf (Computing Generators of the Unit Group)}\label{alg:unit_group_gens}\ \\
Let $R = P/I$ be a finite $\mathbb{Z}$-algebra. Consider the following sequence of instructions.

\begin{algorithmic}[1]
            \STATE Let $G =  [\;]$.
            \STATE Compute the nilradical $\Rad(0)$ of $R$.
            \STATE Apply Algorithm~\ref{alg:units_reduced} to compute elements in $R$ such that their residue classes
                generate the unit group of $R/\Rad(0)$. Add these elements to $G$.
            \STATE Compute the nilpotency index~$s$ of~$\Rad(0)$ in $R$.
            \FOR{$i=2, \dots, s$}
                \STATE Compute $f_1, \dots, f_{\ell} \in R$ such that the residue classes
                $\bar{f}_1, \dots, \bar{f}_{\ell}$ in $R/\Rad(0)^i$ generate~$1+ \Rad(0)^{i-1}$.
                \STATE Add $f_1, \dots, f_{\ell}$ to~$G$.
            \ENDFOR
            \RETURN $G$.
\end{algorithmic}
This is an algorithm which computes a generating set of the unit group $R^\times$.
\end{algo}
    
\begin{proof}
After line~3 has been executed the list $G$ contains generators of the unit group of $R/\Rad(0)$. Now assume
that the elements in $G$ generate $R/\Rad(0)^{i-1}$ after iteration number $i-1$ of lines 5--8. By
Lemma~\ref{lemma:canonical_epi} the group homomorphism
$$
\varphi_i^\times \colon \left(R/\Rad(0)^{i-1}\right)^\times \rightarrow \left(R/\Rad(0)^{i}\right)^\times
$$
is surjective and its kernel is given by $1+\Rad(0)^{i-1}$. This implies
$$
\left(R/\Rad(0)^{i-1}\right)^\times/\left(1+\Rad(0)^{i-1}\right) \cong \left(R/\Rad(0)^i\right)^\times.
$$
The generators of $R/\Rad(0)^{i-1}$ together with the generators of $1+\Rad(0)^{i-1}$ therefore generate
$\left(R/\Rad(0)^i\right)^\times$. Hence, after iteration number $s$ the list $G$ contains generators of
$\left(R/\Rad(0)^s\right)^\times = R^\times$.
\end{proof}

The loop in lines 5--8 can be avoided as follows.

\begin{lem}\label{lemma:rad_Macaulay}
Let $R = P/I$ be a finite $\mathbb{Z}$-algebra, and let $\mathcal{G} = \mathbb{T}^n \setminus L$ be a set of
terms as in Lemma~\ref{lemma:macaulay} such that the residue classes of the elements in $\mathcal{G}$
generate~$R^+$. Let $f_1, \dots, f_k \in P$ such that their residue classes generate $\Rad(0)$. Then
$\mathcal{H} = \left\{ 1+ \bar{t}\bar{f}_i \mid i=1, \dots, k, \; t \in \mathcal{G} \right\}$ generates
$1+\Rad(0)$.
\end{lem}

\begin{proof}
Clearly, $\mathcal{H}$ is contained in $1+\Rad(0)$. By the proof of Algorithm~\ref{alg:unit_group_gens}, it is
then enough to show that there are $g_1, \dots, g_s \in H$ such that $\bar{g}_1, \dots, \bar{g}_s$
in~$R/\Rad(0)^i$ generate $1+ \Rad(0)^{i-1}$. Every element in $\Rad(0)^{i-2}$ can be written as a
$\mathbb{Z}$-linear combination of the residue classes of the terms in $\mathcal{G}$. Thus, every element
in $\Rad(0)^{i-1}$ can be written as a $\mathbb{Z}$-linear combination of the terms in
$\left\{\bar{t}\bar{f}_i \mid i=1, \dots, k, \; t \in \mathcal{G} \right\}$. The claim now follows from the fact
that $(1+\bar{f})(1+\bar{g}) = 1+\bar{f}+\bar{g}$ for $f,g \in \Rad(0)^{i-1}$, since $fg \in \Rad(0)^i$.
\end{proof}

Let us apply Algorithm~\ref{alg:unit_group_gens} using this simplification to a concrete example.

\begin{exa}
Consider the finite $\mathbb{Z}$-algebra $R = \mathbb{Z}[x,y]/I$, where $I=\left\langle x^3, 6x^2, y^2+y+1 \right\rangle$.
\begin{enumerate}
\item[1:] Let $G = [\;]$.

\item[2:] Compute $\Rad(0) = \left\langle \bar{x}, \bar{y}^2+\bar{y}+1 \right\rangle$.

\item[3:] Using Algorithm~\ref{alg:units_reduced}, we compute
                $(R/\Rad(0))^\times = \langle \bar{y}+1 \rangle$ and add $\bar{y}+1$ to $G$.

\item[4-8:] Using Lemma~\ref{lemma:macaulay}, we determine $\mathbb{Z}$-module generators
                $\mathcal{G} = \left\{\bar{x}^2\bar{y}, \bar{x}\bar{y}, \bar{y}, \bar{x}^2, \bar{x}, 1\right\}$ of~$R$.
                Then for every generator $g$ of~$\Rad(0)$ and for every term $t \in \mathcal{G}$ we calculate
                $1+\NF_{\sigma, I}(tg)$ and obtain generators
                $1+\bar{x}, 1+\bar{x}^2, 1+\bar{x}\bar{y}, 1+\bar{x}^2\bar{y}$ of $1+\Rad(0)$.

\item[9:] We obtain
$$
R^\times = \left\langle 1+\bar{y}, 1+\bar{x}, 1+\bar{x}^2, 1+\bar{x}\bar{y}, 1+\bar{x}^2\bar{y} \right\rangle.
$$

\end{enumerate}
\end{exa}

The generating set produced by Algorithm~\ref{alg:unit_group_gens} is in general not minimal. But we can compute
the exponent lattice of these generators to identify redundant ones. Furthermore, the isomorphism type of the
unit group can be determined by computing the Smith normal form of the exponent lattice.

\begin{cor}{\bf (Computing the Isomorphism Type of the Unit Group)}\label{corollary:type}\,\\
Let $R = P/I$ be a finite $\mathbb{Z}$-algebra. Consider the following sequence of instructions.

\begin{algorithmic}[1]
   \STATE Using Algorithm~\ref{alg:unit_group_gens}, compute generators $g_1, \dots, g_k$ of $R^\times$.
   \STATE Using Algorithm~\ref{alg:exp_lattice_z-algebra}, compute generators $v_1, \dots, v_m \in \mathbb{Z}^k$
          which generate the exponent lattice of $(g_1, \dots, g_k)$ in $R$.
   \STATE Form the matrix whose rows are given by $v_1, \dots, v_m$, and compute its Smith normal form~$S$.
   \STATE Let $r$ be the number of diagonal entries of $S$ equal to zero, and let $k_1, \dots, k_u$ be the
          non-zero diagonal entries of $S$.
\RETURN $r$ and $k_1, \dots, k_u$.
\end{algorithmic}
This is an algorithm which computes the rank and the invariant factors of the group~$R^\times$.
\end{cor}

Algorithm~\ref{alg:unit_group_gens} for computing generators of the unit group of a finite $\mathbb{Z}$-algebra
raises the natural question whether it can be extended to finitely generated $\mathbb{Z}$-algebras which are not
necessarily finite $\mathbb{Z}$-modules.
The unit group of such rings is finitely generated if and only if the Jacobson radical is finitely generated as
an additive group (see~\cite[Thm.~1]{bass1974introduction}). In particular, the unit group of a 
finitely generated, reduced $\mathbb{Z}$-algebra is a finitely generated group. 
However, the previous results cannot be directly applied to
compute a set of generators. For example, in the integral case, the algebra no longer needs to be a finite field or
an order in a number field. To the best of our knowledge, there exist no algorithms which compute generators of the
unit group in this case. We leave this problem for future research.

\end{document}